\documentclass[12pt]{amsart}
\usepackage{graphicx}
\usepackage[headings]{fullpage}
\usepackage{amssymb,epic,eepic,epsfig,amsbsy,amsmath,amscd,color}
\numberwithin{equation}{section}
                        \theoremstyle{plain}
\usepackage{mathrsfs}
         
\newcommand\no[1]{}

\newtheorem{theorem}{Theorem}[section]
\newtheorem{thm}{Theorem}
\newtheorem{lemma}[theorem]{Lemma}

\newtheorem{proposition}[theorem]{Proposition}

\theoremstyle{definition}

\newcommand{\lcr}{\raisebox{-5pt}{\mbox{}\hspace{1pt}
                  \epsfig{file=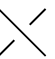}\hspace{1pt}\mbox{}}}

\def\BC{\mathbb C}

\def\BZ{\mathbb Z}
\def\BR{\mathbb R}

\def\la{\langle}
\def\ra{\rangle}

\DeclareMathOperator{\tr}{\mathrm tr}

\def\be { \begin{equation} }
\def\ee { \end{equation} }

\begin{document}
\allowdisplaybreaks
\baselineskip16pt
\title[Left orderable surgeries of double twist knots]{Left orderable surgeries of double twist knots}

\begin{abstract}
A rational number $r$ is called a left orderable slope of a knot $K \subset S^3$ if the 3-manifold obtained from $S^3$ by $r$-surgery along $K$ has left orderable fundamental group. In this paper we consider the double twist knots $C(k,l)$ in the Conway notation. For any positive integers $m$ and $n$, we show that if $K$ is a double twist knot of the form $C(2m,-2n)$,  $C(2m+1, 2n)$ or $C(2m+1, -2n)$ then there is an explicit unbounded interval $I$ such that 
any rational number $r \in I$ is a left orderable slope of $K$. 
\end{abstract}

\author{Anh T. Tran}



\address{Department of Mathematical Sciences,
The University of Texas at Dallas,
Richardson, TX 75080-3021, USA}
\email{att140830@utdallas.edu}

\maketitle

\section{Introduction}

The motivation of this paper is the L-space conjecture of Boyer, Gordon and Watson \cite{BGW} which states that an irreducible rational homology 3-sphere is an L-space if and only if its fundamental group is not left orderable. Here a rational homology 3-sphere $Y$ is an L-space if its Heegaard Floer homology $\widehat{\mathrm{HF}}(Y)$ has rank equal to the order of $H_1(Y; \BZ)$, and a non-trivial group $G$ is left orderable if it admits a total ordering $<$ such that $g<h$ implies $fg < fh$ for all elements $f,g,h$ in $G$.  A knot $K$ in $S^3$ is called an L-space knot if it admits a positive Dehn surgery yielding an L-space. It is known that non-torus alternating knots are not L-space knots, see \cite{OS}. In view of the L-space conjecture, this would imply that any non-trivial Dehn surgery along a non-torus alternating knot produces a 3-manifold with left orderable fundamental group. 

A rational number $r$ is called a left orderable slope of a knot $K \subset S^3$ if the 3-manifold obtained from $S^3$ by $r$-surgery along $K$ has left orderable fundamental group.  As mentioned above, one would expect that any rational number is a left orderable slope of any non-torus alternating knot.  It is known that any rational number $r \in (-4,4)$ is a left orderable slope of the figure eight knot, and any rational number $r \in [0,4]$ is a left orderable slope of  the hyperbolic twist knot $5_2$, see \cite{BGW} and \cite{HT-52} respectively. Consider the double twist knot $C(k,l)$ in the Conway notation as in Figure 1, where $k, l$ denote the numbers of horizontal half-twists with sign in the boxes. Here the sign of $\lcr$ is positive in the box $k$ and  is negative in the box $l$. Then the following results were shown in  \cite{HT-genus1, Tr} by using continuous families of hyperbolic $\mathrm{SL}_2(\BR)$-representations of knot groups. If $m,n$ are integers $\ge 1$, any rational number $r \in (-4n, 4m)$ is a left orderable slope of $C(2m, 2n)$. If $m, n$ are integers $\ge 2$ then any rational number $r \in [0, \max\{4m, 4n\})$ is a left orderable slope of $C(2m, - 2n)$ and any rational number $r \in [0, 4]$ is a left orderable slope of both $C(2m, - 2)$  and $C(2, -2n)$. Note that $C(2,2)$ is the figure eight knot and $C(4,-2)$ is the twist knot $5_2$. Moreover $C(2,-2)$ is the trefoil knot, which is the $(2,3)$-torus knot. 

\begin{figure}[th]
\centerline{\psfig{file=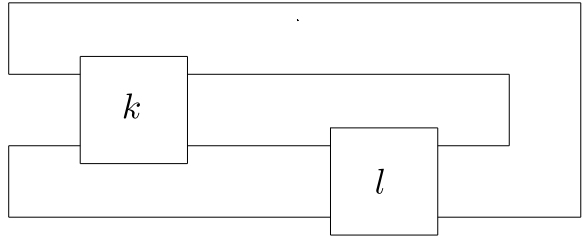,width=3in}}
\caption{The double twist knot/link $C(k,l)$ in the Conway notation.}
\end{figure} 

In this paper, by using continuous families of elliptic $\mathrm{SL}_2(\BR)$-representations of knot groups we extend the range of left orderable slopes of $C(2m, -2n)$. Moreover, we also give left orderable slopes of $C(2m+1, \pm 2n)$. 

 \begin{thm} \label{main}
 Suppose $K$ is a double twist knot of the form $C(2m,-2n)$,  $C(2m+1, 2n)$ or $C(2m+1, -2n)$ in the Conway notation for some positive integers $m$ and $n$. Let 
$$
\mathrm{LO}_K = \begin{cases} (-\infty,1) &\mbox{if } K = C(2m,-2n), \\ 
(-\infty, 2n-1) & \mbox{if } K = C(2m+1,2n),  \\ 
(3-2n, \infty) & \mbox{if } K = C(2m+1,-2n) \text{ and } n \ge 2. \end{cases}
$$
Then any rational number $r \in \mathrm{LO}_K$ is a left orderable slope of  $K$. 
 \end{thm}
 
Combining this with results in \cite{HT-genus1, Tr}, we conclude that if $m$ and $n$ are integers $\ge 2$ then any rational number $r \in (-\infty, \max\{4m, 4n\})$ is a left orderable slope of  $C(2m, - 2n)$ and any rational number $r \in (-\infty, 4]$ is a left orderable slope of both $C(2m, - 2)$  and $C(2, -2n)$. In the subsequent paper \cite{KTT} we will use continuous families of hyperbolic $\mathrm{SL}_2(\BR)$-representations of knot groups to extend the range of left orderable slopes of $C(2m+1, - 2n)$. More specifically, we will show that any rational number $r \in (-4n, 4m)$ is a left orderable slope of $C(2m+1,-2n)$ detected by hyperbolic $\mathrm{SL}_2(\BR)$-representations of the knot group.

We remark that in the case of $C(2m+1, \pm 2n)$, where $m$ and $n$ are positive integers, Gao \cite{Ga} independently obtains similar results. She proves  a weaker result that any rational number $r \in (-\infty, 1)$ is a left orderable slope of  $C(2m+1, 2n)$ and a stronger result that any rational number $r \in (-4n, \infty)$ is a left orderable slope of  $C(2m+1, -2n)$. 

As in \cite{BGW, HT-genus1, HT-52, Tr,  CD} the proof of Theorem \ref{main} is based on the existence of continuous families of elliptic $\mathrm{SL}_2(\BR)$-representations of the knot groups of double twist knots $C(2m,-2n)$ and $C(2m+1, \pm 2n)$ into $\mathrm{SL}_2(\BR)$ and the fact that $\widetilde{\mathrm{SL}_2(\BR)}$, which is the universal covering group of $\mathrm{SL}_2(\BR)$, is a left orderable group. 

This paper is organized as follows. In Section 1, we study certain real roots of the Riley polynomial of double twist knots $C(k,-2p)$, whose zero locus describes all non-abelian representations of the knot group into $\mathrm{SL}_2(\BC)$. In Section 2, we prove  Theorem \ref{main}.

 \section{Real roots of the Riley polynomial}
 
For a knot $K$ in $S^3$, let $G(K)$ denote the knot group of $K$ which is the fundamental group of the complement of an open tubular neighborhood of $K$. 

Consider the double twist knot/link $C(k, l)$ in the Conway notation as in Figure 1, where $k, l$ are integers such that $|kl| \ge 3$. Note that $C(k, l)$ is the rational knot/link corresponding to continued fraction $k+1/l$. It is easy to see that $C(k,l)$ is the mirror image of $C(l,k)=C(-k,-l)$. Moreover, $C(k, l)$ is a knot if $kl$ is even and is a two-component link if $kl$ is odd. In this paper, we only consider knots and so we can assume that $k>0$ and $l=-2p$ is even. 

Note that $C(k, -2p)$ is the mirror image of the double twist knot $J(k,2p)$ in \cite{HS}.  Then, by \cite{HS}, the knot group of $C(k, -2p)$ has a presentation 
$$G(C(k, -2p)) = \la a, b \mid a w^{p}  =  w^{p} b\ra$$ 
where $a, b$ are meridians and $$w = \begin{cases} (ab^{-1})^m(a^{-1}b)^m &\mbox{if } k=2m, \\ 
(ab^{-1})^m ab (a^{-1}b)^m & \mbox{if } k=2m+1. \end{cases}$$
Moreover, the canonical longitude of $C(k, -2p)$ corresponding to the meridian $\mu=a$ is $\lambda = (w^{p} (w^{p})^*a^{-2\varepsilon})^{-1}$, where $\varepsilon =0$ if $k=2m$ and $\varepsilon =2p$ if $k=2m+1$. Here, for a word $u$ in the letters $a,b$ we let $u^*$ be the word obtained by reading $v$ backwards. 

Suppose $\rho: G(C(k, -2p)) \to \mathrm{SL}_2(\BC)$ is a nonabelian representation. 
Up to conjugation, we may assume that 
\begin{equation} \label{repn}
\rho(a) = \left[ \begin{array}{cc}
M& 1 \\
0 & M^{-1} \end{array} \right] \quad \text{and} \quad 
\rho(b) = \left[ \begin{array}{cc}
M & 0 \\
2 - y & M^{-1} \end{array} \right]
\end{equation}
where $(M,y) \in \BC^2$ satisfies the matrix equation $\rho(a w^p) = \rho(w^p b)$. It is known that this matrix equation is equivalent to a single polynomial equation $R_{C(k, -2p)}(x,y) =0$, where $x= (\tr \rho(a))^2$ and $R_K(x,y)$ is the Riley polynomial of $K$, see \cite{Ri}. This polynomial can be described via the Chebychev polynomials as follows. 

Let $\{S_j(v)\}_{j \in \BZ}$ be the Chebychev polynomials in the variable $v$ defined by $S_0(v) = 1$, $S_1(v) = v$ and $S_j(v) = v S_{j-1}(v) - S_{j-2}(v)$ for all integers $j$. Note that $S_{j}(v) = - S_{-j-2}(v)$ and $S_j(\pm2) = (\pm 1)^j (j+1)$. Moreover, we have $S_j(v) = (s^{j+1} - s^{-(j+1)})/(s - s^{-1})$ for $v = s + s^{-1} \not= \pm 2$. Using this identity one can prove the following.

\begin{lemma} \label{chev}
For any integer $j$ and any positive integer $n$ we have 
\begin{enumerate}
\item   $S^2_j(v) - v S_j(v)S_{j-1}(v) + S^2_{j-1}(v) =1$.
\item $S_n(v)-S_{n-1}(v)=\prod_{j=1}^n \big( v-2\cos \frac{(2j-1)\pi}{2n+1} \big)$. 
\item $S_n(v) + S_{n-1}(v)=\prod_{j=1}^n \big( v-2\cos \frac{2j\pi}{2n+1} \big)$.
\item $S_n(v) =\prod_{j=1}^n \big( v-2\cos \frac{j\pi}{n+1} \big)$.
\end{enumerate}
\end{lemma}

The Riley polynomial of $C(k, -2p)$, whose zero locus describes all non-abelian representations of the knot group of $C(k, -2p)$ into $\mathrm{SL}_2(\BC)$, is
$$R_{C(k, -2p)}(x,y)=S_{p}(t) - z S_{p-1}(t)$$
where 
\begin{eqnarray*}
t &=& \tr \rho(w) = \begin{cases} 2+ (y+2-x)(y-2) S^2_{m-1}(y) &\mbox{if } k=2m, \\ 
2- (y+2-x)(S_m(y)-S_{m-1}(y))^2 & \mbox{if } k=2m+1, \end{cases}
\end{eqnarray*}
and
\begin{eqnarray*}
z &=& \begin{cases} 1 + (y+2-x) S_{m-1}(y) (S_{m}(y) - S_{m-1}(y)) &\mbox{if } k=2m, \\ 
1 - (y+2-x) S_{m}(y) (S_{m}(y) - S_{m-1}(y))  & \mbox{if } k=2m+1. \end{cases}
\end{eqnarray*}
Moreover, for the representation $\rho: G(C(k, -2p)) \to \mathrm{SL}_2(\BC)$ of the form \eqref{repn} the image of the canonical longitude $\lambda =  (w^{p} (w^{p})^*a^{-2\varepsilon})^{-1}$ has the form 
$\rho(\lambda) = \left[ \begin{array}{cc}
L & * \\
0 & L^{-1} \end{array} \right]$, where 
$$L = - \frac{M^{-1}(S_m(y) - S_{m-1}(y)) - M(S_{m-1}(y) - S_{m-2}(y))} {M(S_m(y) - S_{m-1}(y)) - M^{-1} (S_{m-1}(y) - S_{m-2}(y))} \quad  \text{if~}  k=2m$$ and  
$$L = - M^{4p} \frac{M^{-1}S_m(y) - M S_{m-1}(y)}{M S_m(y) - M^{-1} S_{m-1}(y)} \quad \text{if~}  k=2m+1.$$
See e.g. \cite{Tr, Pe}. 

Lemmas \eqref{solution}--\eqref{solution3} below describe continuous families of real roots of the Riley polynomials of the double twist knots $C(2m, -2n)$, $C(2m+1, 2n)$ and $C(2m+1, -2n)$ respectively, where $m$ and $n$ are positive integers.

\begin{lemma} \label{solution} There exists a continuous real function $y: [4-1/(mn), 4] \to [2, \infty)$ in the variable $x$ such that 
\begin{itemize}
\item $y(4-1/(mn)) =2$ and 
\item $R_{C(2m,-2n)}(x, y(x))=0$ for all $x \in [4-1/(mn), 4]$.
\end{itemize}
\end{lemma}

\begin{proof}
Let $K=C(2m,-2n)$. We have $R_K(x,y)=S_{n}(t) - z S_{n-1}(t)$
where 
\begin{eqnarray*}
t &=& 2+ (y+2-x)(y-2) S^2_{m-1}(y), \\
z &=& 1 + (y+2-x) S_{m-1}(y) (S_{m}(y) - S_{m-1}(y)).
\end{eqnarray*}
Consider real numbers $x \in [4-1/(mn), 4]$ and $y \in [2, \infty)$. Since $y \ge 2 \ge x-2$, we have $t \ge 2$ and $z \ge 1$. This implies that $zS_{n-1}(t) - S_{n-2}(t) \ge S_{n-1}(t) - S_{n-2}(t)>0$, by Lemma \ref{chev}. The equation $R_K(x,y)=0$ is then equivalent to 
\begin{equation} \label{quadratic}
\big( S_{n}(t) - z S_{n-1}(t)  \big) \big( S_{n-2}(t) - zS_{n-1}(t)\big)=0.
\end{equation}
 
Let $P(x,y)$ denote the left hand side of equation \eqref{quadratic}. By Lemma \ref{chev}, we have $S^2_n(t) - t S_n(t)S_{n-1}(t) + S^2_{n-1}(t)=1$. This can be written as $S_n(t) S_{n-2}(t) = S^2_{n-1}(t) -1$. From this and $S_n(t) + S_{n-2}(t) = t S_{n-1}(t)$ we get
$$
P(x,y) = (z^2 - t z +1) S^2_{n-1}(t) -1.
$$
By a direct calculation, using $S^2_m(y)  + S^2_{m-1}(y) - y S_m(y)S_{m-1}(y) = 1$, we have 
\begin{eqnarray*}
&&z^2 - tz + 1 \\
&=& (z-1)^2 - (t-2) z \\
                   &=& (y+2-x)^2 S^2_{m-1}(y) (S_{m}(y) - S_{m-1}(y))^2 \\
                   && - \, (y+2-x)(y-2) S^2_{m-1}(y) \big[ 1 + (y+2-x) S_{m-1}(y) (S_{m}(y) - S_{m-1}(y)) \big]\\                 
                   &=& (y+2-x) S^2_{m-1}(y)  \big[ 4-x +  (y+2-x) (y-2) S^2_{m-1}(y) \big] \\
                   &=& (y+2-x) S^2_{m-1}(y) (t+2-x).\end{eqnarray*}
Hence $P(x,y) = (y+2-x) S^2_{m-1}(y)  (t+2-x)S^2_{n-1}(t)-1$. 

By Lemma \ref{chev}(4), for any positive integer $l$ the Chebychev polynomial $S_l(v) =\prod_{j=1}^l (v-2\cos \frac{j\pi}{l+1})$ is a strictly increasing function in $v\in [2,\infty)$. This implies that, for a fixed real number $x \in [4-1/(mn), 4]$, the polynomials $t = 2+ (y+2-x)(y-2) S^2_{m-1}(y) \ge 2$ and $P(x,y) = (y+2-x) S^2_{m-1}(y)  (t+2-x)S^2_{n-1}(t)-1$ are strictly increasing functions in $y \in [2,\infty)$. Note that $\lim_{y \to \infty} P(x,y) = \infty$ and 
$$\lim_{y \to 2^+} P(x,y) = P(x,2) = (4-x)^2 m^2n^2 -1 \le 0.$$
Hence  there exists a unique real number $y(x) \in [2, \infty)$ such that $P(x,y(x))=0$. Since $P(4-\frac{1}{mn}, 2)=0$ we have 
$y(4-\frac{1}{mn}) =2$. Finally, by the implicit function theorem $y = y(x)$ is a continuous function in  $x \in[4-1/(mn), 4]$.
 \end{proof}

\begin{lemma} \label{solution2}
There exists a continuous real function $x: [2, \infty) \to (4 \cos^2 \frac{(2n-1)\pi}{4n+2}, \infty)$ in the variable $y$ such that 
\begin{itemize}
\item $x(2) < 4 \cos^2 \frac{(2n-2)\pi}{4n+2}$, 
\item $\lim_{y \to \infty} x(y) = \infty$ and 
\item $R_{C(2m+1,2n)}(x(y),y)=0$ for all $y \in [2, \infty)$.
\end{itemize}
\end{lemma}

\begin{proof}
Let $K=C(2m+1,2n)$. We have $R_K(x,y)= S_{-n}(t) - z S_{-n-1}(t)$ where 
\begin{eqnarray*}
t &=&2- (y+2-x)(S_m(y)-S_{m-1}(y))^2,\\
z &=& 1 - (y+2-x) S_{m}(y) (S_{m}(y) - S_{m-1}(y)).
\end{eqnarray*} 
Note that $R_K(x,y)= (t - z) S_{-n-1}(t) - S_{-n-2}(t) = S_{n}(t) - (t-z) S_{n-1}(t)$. 

By Lemma \ref{chev} we have 
\begin{eqnarray*}
S_n(t)-S_{n-1}(t) &=& \prod_{j=1}^n \Big( t-2\cos \frac{(2j-1)\pi}{2n+1} \Big), \\
S_n(t) + S_{n-1}(t) &=& \prod_{j=1}^n \Big( t-2\cos \frac{2j\pi}{2n+1} \Big).
\end{eqnarray*}

Let $t_j=2\cos \frac{j\pi}{2n+1}$ for $j = 1, \cdots, 2n$. By writing $t_{2j-1}=e^{i\theta} + e^{-i\theta}$ where $\theta = \frac{(2j-1)\pi}{2n+1}$, we have 
\begin{eqnarray*}
S_n(t_{2j-1}) &=& \frac{e^{i(n+1)\theta} - e^{-i(n+1)\theta}}{e^{i\theta} - e^{-i\theta}} = \frac{\sin \frac{(2j-1)(n+1)\pi}{2n+1}} {\sin \frac{(2j-1)\pi}{2n+1}} \\
&=& \frac{\sin \Big( j\pi -\frac{\pi}{2}  + \frac{(2j-1)\pi}{2(2n+1)} \Big)} {\sin \frac{(2j-1)\pi}{2n+1}}=(-1)^{j-1}\frac{\cos \frac{(2j-1)\pi}{2(2n+1)} } {\sin \frac{(2j-1)\pi}{2n+1}}.
\end{eqnarray*}
This implies that $(-1)^{j-1}S_n(t_{2j-1})>0$. Similarly, $(-1)^{j}S_n(t_{2j})>0$.

Fix a real number $y  \ge 2$.  Let $s_j(y) = y+2 - \frac{2-t_j}{(S_m(y)-S_{m-1}(y))^2}$ for $j = 1, \cdots, 2n$. We also let $s_0 = y+2$. Since $-2 < t_{2n} < \cdots < t_1 < 2$ we have $s_{2n}(y) < \cdots < s_1(y)< y+2 = s_0(y)$. At $x=s_{2j-1}(y)$ we have $t=t_{2j-1}$ and so $S_n(t)=S_{n-1}(t)$. This implies that 
\begin{eqnarray*}
R_K(s_{2j-1}(y),y) &=& (1-(t-z))S_n(t_{2j-1}) \\
&=&-(y+2-s_{2j-1}(y))S_{m-1}(y) ( S_{m}(y) - S_{m-1}(y) )S_n(t_{2j-1}).
\end{eqnarray*}
Since $y \ge 2$, by Lemma \ref{chev} we have $S_m(y) - S_{m-1}(y)\ge S_m(2) - S_{m-1}(2) = 1$ and $S_{m-1}(y) \ge S_{m-1}(2)  = m$. Hence $(-1)^{j} R_K(s_{2j-1}(y),y)>0$. 

Similarly, for $1 \le j \le n$ we have
\begin{eqnarray*}
R_K(s_{2j}(y),y) &=& (1+t-z)S_n(t_{2j}) \\
&=& \big[  2 + (y+2-s_{2j-1}(y))S_{m-1}(y) (S_{m}(y) - S_{m-1}(y) \big] S_n(t_{2j}),
\end{eqnarray*}
which implies that $(-1)^{j} R_K(s_{2j}(y),y)>0$.

For each $1 \le j \le n-1$, since $$R_K(s_{2j+1}(y),y) R_K(s_{2j}(y),y)<0$$ there exists $x_j(y) \in (s_{2j+1}(y), s_{2j}(y))$ such that $R_K(x_j(y), y)=0.$ Since $$R_K(s_0(y),y) =  R_K(y+2,y)=1$$ and $R_K(s_1(y), y) < 0$ there exists $x_0(y) \in (s_1(y), s_0(y))$ such that $R_K(x_0(y), y)=0$. 

Since $R_K(x,y) = z S_{n-1}(t) - S_{n-2}(t)$, 
we see that $R_K(x,y)$ is a polynomial of degree $n$ in $x$ for each fixed real number $y \ge 2$. This polyomial has exactly $n$  simple real roots $x_0(y), \cdots, x_{n-1}(y)$ satisfying $x_{n-1}(y) <  \cdots <  x_0(y) < y+2$, hence the implicit function theorem implies that each $x_j(y)$ is a continuous function in $y  \ge 2$. 

By letting $x(y) = x_{n-1}(y)$ for $y \ge 2$, we have $R_K(x(y),y)=0$. Moreover, since $$x(y) > s_{2n-1}(y)= y+2 - \frac{2-2\cos \frac{(2n-1)\pi}{2n+1}}{(S_m(y)-S_{m-1}(y))^2}$$ we have $\lim_{y \to \infty} x(y) = \infty$ and $x(y) > 4- \big( 2-2\cos \frac{(2n-1)\pi}{2n+1} \big) = 4 \cos^2 \frac{(2n-1)\pi}{4n+2}$ for $y \ge 2$. 

Finally, since $x(y) < s_{2n-2}(y)$ for all $y \ge 2$ we have $x(2) < s_{2n-2}(2) = 4 \cos^2 \frac{(2n-2)\pi}{4n+2}$.
\end{proof}

\begin{lemma} \label{solution3}
Suppose $n \ge 2$. Then there exists a continuous real function $x: [2, \infty) \to (4 \cos^2 \frac{(2n-1)\pi}{4n+2}, \infty)$ in the variable $y$ such that 
\begin{itemize}
\item $x(2) < 4 \cos^2 \frac{(2n-3)\pi}{4n+2}$, 
\item $\lim_{y \to \infty} x(y) = \infty$ and 
\item $R_{C(2m+1,-2n)}(x(y),y)=0$ for all $y \in [2, \infty)$.
\end{itemize}
\end{lemma}

\begin{proof}
Let $K=C(2m+1,-2n)$. We have $R_{K}(x,y)= S_n(t) -  z S_{n-1}(t)$ where 
\begin{eqnarray*}
t &=& 2- (y+2-x)(S_m(y)-S_{m-1}(y))^2,\\
z &=& 1 - (y+2-x)S_m(y) ( S_m(y)-S_{m-1}(y) ).
\end{eqnarray*}

 Fix a real number $y  \ge 2$. Choose $t_j$ and $s_j(y)$ for $1 \le j \le 2n$ as in Lemma \ref{solution2}. Since 
 \begin{eqnarray*}
R_K(s_{2j-1}(y),y) &=& (1-z)S_n(t_{2j-1}) \\
&=& (y+2-s_{2j-1}(y))S_m(y) ( S_{m}(y) - S_{m-1}(y) )S_n(t_{2j-1}),
\end{eqnarray*}
we have  $(-1)^{j-1}R_K(s_{2j-1}(y),y)>0$. Hence, there exists $x_j(y) \in (s_{2j+1}(y), s_{2j-1}(y))$ such that $R_K(x_j(y), y)=0$ for each $1 \le j \le n-1$. 

By writing $R_K(x,y) = (t-z) S_{n-1}(t) - S_{n-2}(t)$ and noting that
\begin{eqnarray*}
t-z = 1+ (y+2-x)( S_{m}(y) - S_{m-1}(y)) S_{m-1}(y),
\end{eqnarray*}
we see that $R_K(x,y)$ is a polynomial of degree $n$ in $x$ with negative highest coefficient for each fixed real number $y \ge 2$. Since $\lim_{x \to \infty} R_K(x,y) = -\infty$ and $R_K(y+2,y)=1$, there exists $x_0(y) \in (y+2, \infty)$ such that $R_K(x_0(y),y)=0$. For a fixed real number $y \ge 2$, the polynomial $R_K(x,y)$ of degree $n$ in $x$ has exactly $n$  simple real roots $x_0(y), \cdots, x_{n-1}(y)$ satisfying $x_{n-1}(y) <  \cdots <  x_1(y) < y+2 < x_0(y) $, hence the implicit function theorem implies that each $x_j(y)$ is a continuous function in $y  \ge 2$. 

By letting $x(y) = x_{n-1}(y)$ for $y \ge 2$, we have $R_K(x(y),y)=0$. Moreover, since $$x(y) > s_{2n-1}(y)= y+2 - \frac{2-2\cos \frac{(2n-1)\pi}{2n+1}}{(S_m(y)-S_{m-1}(y))^2}$$ we have $\lim_{y \to \infty} x(y) = \infty$ and $x(y) > 4- \big( 2-2\cos \frac{(2n-1)\pi}{2n+1} \big) = 4 \cos^2 \frac{(2n-1)\pi}{4n+2}$ for $y \ge 2$. 

Finally, since $x(y) < s_{2n-3}(y)$ for all $y \ge 2$ we have $x(2) < s_{2n-3}(2) = 4 \cos^2 \frac{(2n-3)\pi}{4n+2}$.
\end{proof}

 \section{Proof of Theorem \ref{main}}

Suppose $K$ is a double twist knot of the form $C(2m,-2n)$,  $C(2m+1, 2n)$ or $C(2m+1, -2n)$ in the Conway notation for some positive integers $m$ and $n$. Let $X$ be the complement of an open tubular neighborhood of $K$ in $S^3$, and $X_r$ the 3-manifold obtained from $S^3$ by $r$-surgery along $K$. Recall that
$$
\text{LO}_K = \begin{cases} (-\infty,1) &\mbox{if } K = C(2m,-2n), \\ 
(-\infty, 2n-1) & \mbox{if } K = C(2m+1,2n),  \\ 
(3-2n, \infty) & \mbox{if } K = C(2m+1,-2n) \text{ and } n \ge 2. \end{cases}
$$

An element of $\mathrm{SL}_2(\BR)$ is called elliptic if its trace is a real number in $(-2,2)$. A representation $\rho: \BZ^2 \to \mathrm{SL}_2(\BR)$ is called elliptic if the image group $\rho(\BZ^2)$ contains an elliptic element of $\mathrm{SL}_2(\BR)$. In which case, since $\BZ^2$ is an abelian group every non-trivial element of $\rho(\BZ^2)$ must also be elliptic.

Using Lemmas \ref{solution}--\ref{solution3} we first prove the following.

\begin{proposition} \label{prop} 
For each rational number $r \in \mathrm{LO}_K \setminus \{0\}$ there exists a  representation $\rho: \pi_1(X_r) \to \mathrm{SL}_2(\BR)$ such that $\rho\big|_{\pi_1(\partial X)}: \pi_1(\partial X) \cong \BZ^2 \to \mathrm{SL}_2(\BR)$ is an elliptic representation. 
\end{proposition}

\begin{proof}
We first consider the case $K=C(2m,-2n)$. Let $\theta_0 = \arccos \sqrt{1-1/(4mn)}$. For $\theta \in (0, \theta_0) \cup (\pi - \theta_0, \pi)$ we let $x=4\cos^2\theta$. Then $x \in (4-1/(mn), 4)$. Consider the continuous real function $$y:  [4-1/(mn), 4] \to [2, \infty)$$ in Lemma \ref{solution}. Let $M=e^{i\theta}$. Then $x = 4 \cos^2 \theta = (M+M^{-1})^2$. Since $R_K(x,y(x))=0$ there exists a non-abelian representation $\rho: \pi_1(X) \to \mathrm{SL}_2(\BC)$ such that $$\rho(a) = \left[ \begin{array}{cc}
M& 1 \\
0 & M^{-1} \end{array} \right] \quad \text{and} \quad 
\rho(b) = \left[ \begin{array}{cc}
M & 0 \\
2 - y(x) & M^{-1} \end{array} \right].$$ 
Note that $x$ is the square of the trace of a meridian.
Moreover, the image of the canonical longitude $\lambda$ corresponding to the meridian $\mu=a$ has the form $\rho(\lambda) = \left[ \begin{array}{cc}
L & * \\
0 & L^{-1} \end{array} \right]$, where $$
L
= 
- \frac{M^{-1}\alpha - M\beta}{M\alpha - M^{-1} \beta}
$$ and $\alpha = S_m(y(x)) - S_{m-1}(y(x))$, $\beta = S_{m-1}(y(x)) - S_{m-2}(y(x))$. Note that $\alpha > \beta >0$, since $y(x) > 2$.

It is easy to see that $|L| = \sqrt{L \bar{L} }=1$, where $\bar{L}$ denotes the complex conjugate of $L$. Moreover, by a direct calculation, we have
\begin{eqnarray*}
\text{Re}(L) &=&  \big( 2\alpha \beta -(\alpha^2+\beta^2) \cos 2\theta \big )/|M\alpha - M^{-1}\beta|^2,\\
\text{Im}(L) &=& (\alpha^2 - \beta^2)\sin2\theta/|M\alpha - M^{-1}\beta|^2.
\end{eqnarray*}
Note that $\text{Im}(L) >0$ if $\theta \in (0, \theta_0)$ and $\text{Im}(L) <0$ if $\theta \in (\pi - \theta_0, \pi)$. Let 
$$\varphi(\theta) = \begin{cases} \arccos \big[ \big( 2\alpha \beta -(\alpha^2+\beta^2) \cos 2\theta \big)/| e^{i\theta} \alpha - e^{-i\theta} \beta|^2 \big] &\mbox{if } \theta \in (0, \theta_0), \\ 
-\arccos \big[ \big( 2\alpha \beta -(\alpha^2+\beta^2) \cos 2\theta \big)/| e^{i\theta} \alpha - e^{-i\theta} \beta|^2 \big] & \mbox{if } \theta \in (\pi - \theta_0, \pi). \end{cases}$$
Then $L = e^{i\varphi(\theta)}$. Note that $\varphi(\theta) \in (0, \pi)$ if $\theta \in (0, \theta_0)$ and $\varphi(\theta) \in (-\pi, 0)$ if $\theta \in (\pi - \theta_0, \pi)$. 

The function $f(\theta) := - \frac{\varphi(\theta)}{\theta}$ is a continuous function on each of the intervals $(0, \theta_0)$ and $(\pi - \theta_0, \pi)$. As $\theta \to 0^+$ we have $M \to 1$ and $L = - \frac{M^{-1}\alpha - M\beta}{M\alpha - M^{-1} \beta} \to -1$, so $\varphi(\theta) \to \pi$. As $\theta \to \theta_0^-$ we have $x \to 4-1/(mn)$, $y(x) \to 2$ and $\alpha, \beta \to 1$, so $L = - \frac{M^{-1}\alpha - M\beta}{M\alpha - M^{-1} \beta} \to 1$ and $\varphi(\theta) \to 0$. This implies that
$$
\lim_{\theta \to 0^+} - \frac{\varphi(\theta)}{\theta} = -\infty \qquad \text{and} \qquad \lim_{\theta \to \theta_0^-} - \frac{\varphi(\theta)}{\theta} = 0.
$$
Hence the image of $f(\theta)$ on the interval $(0, \theta_0)$ contains the interval $(-\infty,0)$. 

Similarly, since
$$
\lim_{\theta \to (\pi - \theta_0)^+} - \frac{\varphi(\theta)}{\theta}  = 0 \qquad \text{and} \qquad \lim_{\theta \to \pi^-} - \frac{\varphi(\theta)}{\theta} = 1,
$$
 the image of  $f(\theta)$ on the interval $(\pi - \theta_0, \pi)$ contains the interval $(0,1)$. 

Suppose $r=\frac{p}{q}$ is a  rational number such that $r \in (-\infty,0) \cup (0,1)$. Then $r= f(\theta) = - \frac{\varphi(\theta)}{\theta}$ for some $\theta \in (0, \theta_0) \cup (\pi - \theta_0, \pi)$. Since $M^p L^q = e^{i(p\theta + q \varphi(\theta))} = 1$, we have $\rho(\mu^p \lambda^q) = I$. This means that the non-abelian representation $\rho: \pi_1(X) \to \mathrm{SL}_2(\BC)$ extends to a  representation $\rho: \pi_1(X_r) \to \mathrm{SL}_2(\BC)$. Finally, since $2-y(x)<0$, a result in \cite[page 786]{Kh} implies that $\rho$ can be conjugated to an $\mathrm{SL}_2(\BR)$-representation. Note that the restriction of this representation to the peripheral subgroup $\pi_1(\partial X)$ of the knot group is an elliptic representation. This completes the proof of Proposition \ref{prop} for $K=C(2m,-2n)$. 

We now consider the case $K=C(2m+1,2n)$. Consider the continuous real function $$x: [2, \infty) \to \Big( 4 \cos^2 \frac{(2n-1)\pi}{4n+2}, \infty \Big)$$ in Lemma \ref{solution2}. Since $x(2) < 4 \cos^2 \frac{(2n-2)\pi}{4n+2}$ and $\lim_{y \to \infty} x(y) = \infty$, there exists $y^*>2$ such that $x(y^*)=4$ and $4 \cos^2 \frac{(2n-1)\pi}{4n+2} < x (y)< 4$ for all $y \in [2,y^*)$. 

For each $y \in [2,y^*)$ we let $\theta(y) = \arccos(\sqrt{x(y)}/2)$. Then $\theta(2) > \frac{(2n-2)\pi}{4n+2}$, and for $y \in [2,y^*)$ we have $0 < \theta(y) < \frac{(2n-1)\pi}{4n+2}$ and $x(y)=4\cos^2\theta(y)$. Since $R_K(x(y),y)=0$  there exists a non-abelian representation $\rho: \pi_1(X) \to \mathrm{SL}_2(\BC)$ such that $$\rho(a) = \left[ \begin{array}{cc}
M& 1 \\
0 & M^{-1} \end{array} \right] \quad \text{and} \quad 
\rho(b) = \left[ \begin{array}{cc}
M & 0 \\
2 - y & M^{-1} \end{array} \right],$$
where $M=e^{i\theta(y)}$. 
Moreover, the image of the canonical longitude $\lambda$ corresponding to the meridian $\mu=a$ has the form $\rho(\lambda) = \left[ \begin{array}{cc}
L & * \\
0 & L^{-1} \end{array} \right]$, where $$
L
= 
- M^{-4n} \frac{M^{-1}\gamma - M\delta}{M\gamma - M^{-1} \delta}
$$ and $\gamma = S_m(y)$, $\delta = S_{m-1}(y)$. Note that $\gamma > \delta >0$, since $y > 2$. 

As in the previous case,  we write $L=e^{i\varphi(y)}$ where $$\varphi(y) = (2n-2) \pi-4n \theta(y) + \arccos \big[ \big( 2\gamma \delta -(\gamma^2+\delta^2) \cos 2\theta(y) \big)/| e^{i\theta(y)} \gamma - e^{-i\theta(y)} \delta |^2 \big].$$ 
Since $\frac{(2n-2)\pi}{4n+2} < \theta(2) < \frac{(2n-1)\pi}{4n+2}$ we have $-\frac{2\pi}{2n+1}<\varphi(2) <2\pi -  \frac{3\pi}{2n+1}$. 

As $y \to 2^+$, $\rho$ approaches a reducible representation and so $L \to 1, \, \varphi(y) \to \varphi(2) = k2\pi$ for some integer $k$. Since $-\frac{2\pi}{2n+1}<\varphi(2) <2\pi -  \frac{3\pi}{2n+1}$, we must have $\varphi(2) =0$. As $y \to (y^*)^-$, we have $x(y) \to 4$, $M \to 1, \, L = 
- M^{-4n} \frac{M^{-1}\gamma - M\delta}{M\gamma - M^{-1} \delta}\to -1$ and hence $\theta(y) \to 0^+, \, \varphi(y) \to (2l-1)\pi$ for some integer $l$. Since 
\begin{eqnarray*} 
 (2l-1)\pi 
&=& \lim_{y \to (y^*)^-} (2n-2)\pi -4n \theta(y) \\
&& \qquad  + \, \arccos \big[ \big( 2\gamma \delta -(\gamma^2+\delta^2) \cos 2\theta(y) \big)/| e^{i\theta(y)} \gamma - e^{-i\theta(y)} \delta |^2 \big] \\
&=& \lim_{y \to (y^*)^-} (2n-2)\pi \\
&& \qquad  + \, \arccos \big[ \big( 2\gamma \delta -(\gamma^2+\delta^2) \cos 2\theta(y) \big)/| e^{i\theta(y)} \gamma - e^{-i\theta(y)} \delta |^2 \big],
\end{eqnarray*}
we have $(2n-2)\pi \le (2l-1) \pi \le (2n-1)\pi$.  This implies that $2l-1=2n-1$ and $\varphi(y) \to (2n-1)\pi$ as $y \to (y^*)^-$. Hence the image of $g(y) := - \frac{\varphi(y)}{\theta(y)}$ on the interval $(2, y^*)$ contains the interval $(-\infty,0)$. 

Similarly, with $\theta_1(y) = \pi - \theta(y)$ we have $x(y) = 4\cos^2(\theta_1(y))$ and hence for each $y \in [2,y^*)$ there exists a non-abelian representation $\rho_1: \pi_1(X) \to \mathrm{SL}_2(\BC)$ such that $$\rho_1(a) = \left[ \begin{array}{cc}
M& 1 \\
0 & M^{-1} \end{array} \right] \quad \text{and} \quad 
\rho_1(b) = \left[ \begin{array}{cc}
M & 0 \\
2 - y & M^{-1} \end{array} \right],$$
where $M=e^{i \theta_1(y)}$. 
Moreover, the image of the canonical longitude $\lambda$ corresponding to the meridian $\mu=a$ has the form $\rho_1(\lambda) = \left[ \begin{array}{cc}
L & * \\
0 & L^{-1} \end{array} \right]$, where $L=e^{i\varphi_1(y)}$ and 
\begin{eqnarray*}
 \varphi_1(y) 
&=& - (2n-2)\pi + 4n\pi- 4n \theta_1(y) \\
&& - \, \arccos \big[ \big( 2\gamma \delta -(\gamma^2+\delta^2) \cos 2\theta_1(y) \big)/| e^{i\theta_1(y)} \gamma - e^{-i\theta_1(y)} \delta |^2 \big] \\
 &=& - (2n-2)\pi + 4n \theta(y) \\
 && - \,  \arccos \big[ \big( 2\gamma \delta -(\gamma^2+\delta^2) \cos 2\theta_1(y)) \big/| e^{i\theta_1(y)} \gamma - e^{-i\theta_1(y)} \delta |^2 \big].
\end{eqnarray*}
Since $\frac{(2n-2)\pi}{4n+2} < \theta(2) < \frac{(2n-1)\pi}{4n+2}$ we have $-2\pi +  \frac{3\pi}{2n+1}<\varphi_1(2) < \frac{2\pi}{2n+1}$. 

As $y \to 2^+$,  $\rho_1$ approaches a reducible representation and so $L \to 1, \, \varphi_1(y) \to 0$. As $y \to (y^*)^-$, we have $x(y) \to 4$, $M \to -1, \, L = 
- M^{-4n} \frac{M^{-1}\gamma - M\delta}{M\gamma - M^{-1} \delta}\to -1$ and hence $\theta_1(y) \to \pi, \, \varphi_1(y) \to -(2n-1)\pi$. This implies that the image of $g_1(y) := - \frac{\varphi_1(y)}{\theta_1(y)}$ on the interval $(2, y^*)$ contains the interval $(0,2n-1)$. 

The rest of the proof of Proposition \ref{prop} for $C(2m+1,2n)$ is similar to that for $C(2m,-2n)$.

Lastly, we consider the case $K=C(2m+1,-2n)$ and $n \ge 2$. Consider the continuous real function $$x: [2, \infty) \to \Big( 4 \cos^2 \frac{(2n-1)\pi}{4n+2}, \infty \Big)$$ in Lemma \ref{solution3}. Since $x(2) < 4 \cos^2 \frac{(2n-3)\pi}{4n+2}$ and $\lim_{y \to \infty} x(y) = \infty$, there exists $y^*>2$ such that $x(y^*)=4$ and $4 \cos^2 \frac{(2n-1)\pi}{4n+2} < x (y)< 4$ for all $y \in [2,y^*)$. 

For each $y \in [2,y^*)$ we let $\theta(y) = \arccos(\sqrt{x(y)}/2)$. Then $\theta(2) > \frac{(2n-3)\pi}{4n+2}$, and for $y \in [2,y^*)$ we have $0 < \theta(y) < \frac{(2n-1)\pi}{4n+2}$ and $x(y)=4\cos^2\theta(y)$. Since $R_K(x(y),y)=0$  there exists a non-abelian representation $\rho: \pi_1(X) \to \mathrm{SL}_2(\BC)$ such that $$\rho(a) = \left[ \begin{array}{cc}
M& 1 \\
0 & M^{-1} \end{array} \right] \quad \text{and} \quad 
\rho(b) = \left[ \begin{array}{cc}
M & 0 \\
2 - y & M^{-1} \end{array} \right],$$
where $M=e^{i\theta(y)}$. 
Moreover, the image of the canonical longitude $\lambda$ corresponding to the meridian $\mu=a$ has the form $\rho(\lambda) = \left[ \begin{array}{cc}
L & * \\
0 & L^{-1} \end{array} \right]$, where $$
L
= 
- M^{4n} \frac{M^{-1}\gamma - M\delta}{M\gamma - M^{-1} \delta}
$$ and $\gamma = S_m(y)$, $\delta = S_{m-1}(y)$. Note that $\gamma > \delta >0$, since $y > 2$. 

As above,  we write $L=e^{i\varphi(y)}$ where $$\varphi(y) = -(2n-2)\pi + 4n \theta(y) + \arccos \big[ \big( 2\gamma \delta -(\gamma^2+\delta^2) \cos 2\theta(y) \big)/| e^{i\theta(y)} \gamma - e^{-i\theta(y)} \delta |^2 \big].$$ 
Since $\frac{(2n-3)\pi}{4n+2} < \theta(2) < \frac{(2n-1)\pi}{4n+2}$ we have $-2\pi + \frac{4\pi}{2n+1}<\varphi(2) < 2\pi - \frac{(2n-1)\pi}{2n+1}$. 

As $y \to 2^+$, $\rho$ approaches a reducible representation and so $L \to 1, \, \varphi(y) \to \varphi(2) = k2\pi$ for some integer $k$. Since $-2\pi + \frac{4\pi}{2n+1}<\varphi(2) < 2\pi - \frac{(2n-1)\pi}{2n+1}$, we must have $\varphi(2) =0$. 

As $y \to (y^*)^-$, we have $x(y) \to 4$, $M \to 1, \, L= 
- M^{4n} \frac{M^{-1}\gamma - M\delta}{M\gamma - M^{-1} \delta} \to -1$ and hence $\theta(y) \to 0^+, \, \varphi(y) \to (2l-1)\pi$ for some integer $l$. Since 
\begin{eqnarray*}
 (2l-1)\pi &=& \lim_{y \to (y^*)^-} -(2n-2)\pi +4n \theta(y) \\
 && \qquad  + \, \arccos \big[ \big( 2\gamma \delta -(\gamma^2+\delta^2) \cos 2\theta(y) \big)/| e^{i\theta(y)} \gamma - e^{-i\theta(y)} \delta |^2 \big] \\
&=& \lim_{y \to (y^*)^-} -(2n-2)\pi  \\
&& \qquad  + \, \arccos \big[ \big( 2\gamma \delta -(\gamma^2+\delta^2) \cos 2\theta(y) \big)/| e^{i\theta(y)} \gamma - e^{-i\theta(y)} \delta |^2 \big],
\end{eqnarray*}
we have $-(2n-2)\pi \le (2l-1) \pi \le -(2n-3)\pi$.  This implies that $2l-1=-(2n-3)$ and $\varphi(y) \to -(2n-3)\pi$ as $y \to (y^*)^-$. Hence the image of $h(y) := - \frac{\varphi(y)}{\theta(y)}$ on the interval $(2, y^*)$ contains the interval $(0, \infty)$. 

Similarly, with $\theta_1(y) = \pi - \theta(y)$ we have $x(y) = 4\cos^2(\theta_1(y))$ and hence for each $y \in [2,y^*)$ there exists a non-abelian representation $\rho_1: \pi_1(X) \to \mathrm{SL}_2(\BC)$ such that $$\rho_1(a) = \left[ \begin{array}{cc}
M& 1 \\
0 & M^{-1} \end{array} \right] \quad \text{and} \quad 
\rho_1(b) = \left[ \begin{array}{cc}
M & 0 \\
2 - y & M^{-1} \end{array} \right],$$
where $M=e^{i \theta_1(y)}$. 
Moreover, the image of the canonical longitude $\lambda$ corresponding to the meridian $\mu=a$ has the form $\rho_1(\lambda) = \left[ \begin{array}{cc}
L & * \\
0 & L^{-1} \end{array} \right]$, where $L=e^{i\varphi_1(y)}$ and 
\begin{eqnarray*}
 \varphi_1(y) &=& (2n-2)\pi - 4n\pi +  4n \theta_1(y) \\
 && - \, \arccos \big[ \big( 2\gamma \delta -(\gamma^2+\delta^2) \cos 2\theta_1(y) \big)/| e^{i\theta_1(y)} \gamma - e^{-i\theta_1(y)} \delta |^2 \big] \\
 &=& (2n-2)\pi - 4n \theta(y) \\
 && - \,  \arccos \big[ \big( 2\gamma \delta -(\gamma^2+\delta^2) \cos 2\theta_1(y) \big)/| e^{i\theta_1(y)} \gamma - e^{-i\theta_1(y)} \delta |^2 \big].
\end{eqnarray*}
Since $\frac{(2n-3)\pi}{4n+2} < \theta(2) < \frac{(2n-1)\pi}{4n+2}$ we have 
$-2\pi + \frac{(2n-1)\pi}{2n+1} <\varphi_1(2) < 2\pi - \frac{4\pi}{2n+1}$. 

As $y \to 2^+$,  $\rho_1$ approaches a reducible representation and so $L \to 1, \, \varphi_1(y) \to \varphi_1(2) =0$. As $y \to (y^*)^-$, we have $x(y) \to 4$, $M \to -1, \, L = 
- M^{4n} \frac{M^{-1}\gamma - M\delta}{M\gamma - M^{-1} \delta} \to -1$ and hence $\theta_1(y) \to \pi, \, \varphi_1(y) \to (2n-3)\pi$. This implies that the image of $h_1(y) := - \frac{\varphi_1(y)}{\theta_1(y)}$ on the interval $(2, y^*)$ contains the interval $( -(2n-3),0)$. 

The rest of the proof of Proposition \ref{prop} for $C(2m+1,-2n)$ is similar to that for $C(2m,-2n)$.
\end{proof}

We now finish the proof of Theorem \ref{main}. Suppose $r$ is a rational number such that $r \in \text{LO}_K$. If $r \not= 0$, by Proposition \ref{prop}, there exists a  representation $\rho: \pi_1(X_r) \to \mathrm{SL}_2(\BR)$ such that $\rho\big|_{\pi_1(\partial X)}$ is an elliptic representation. This representation lifts to a representation $\tilde{\rho}: \pi_1(X_r) \to \widetilde{\mathrm{SL}_2(\BR)}$, where $\widetilde{\mathrm{SL}_2(\BR)}$ is the universal covering group of $\mathrm{SL}_2(\BR)$. See e.g. \cite[Sec. 3.5]{CD} and \cite[Sec. 2.2]{Va}. Note that $X_r$ is an irreducible 3-manifold (by \cite{HTh}) and $\widetilde{\mathrm{SL}_2(\BR)}$ is a left orderable group (by \cite{Be}). Hence, by \cite{BRW}, $\pi_1(X_r)$ is a left orderable group. 
Finally, $0$-surgery along a knot always produces a prime manifold whose first Betti number is $1$, and by \cite{BRW} such manifold has left orderable fundamental group.

\section*{Acknowledgements} 
The author has been partially supported by a grant from the Simons Foundation (\#354595). He would like to thank the referees for helpful comments and suggestions.

\end{document}